\documentclass[a4paper,leqno,12pt]{amsart}
\usepackage[leqno]{amsmath}
\usepackage{amstext,amssymb,amsthm,enumitem}
\usepackage{graphicx}
\usepackage{tikz}
\usepackage[cp1250]{inputenc}
\usepackage{float}
\usepackage{mathtools}
\newcommand{\subscript}[2]{$#1 #2$}

\newtheorem{lemma}{Lemma}
\newtheorem*{theorem*}{Theorem}
\newtheorem{theorem}{Theorem}
\newtheorem{remark}{Remark}
\newtheorem{proposition}{Proposition}

\theoremstyle{definition}
\newtheorem*{proof1*}{Proof of Theorem~\ref{T1}}
\newtheorem*{proof2*}{Proof of Theorem~\ref{T2}}

\newcommand{\NN}{\mathbb{N}}
\newcommand{\ZZ}{\mathbb{Z}}
\newcommand{\RR}{\mathbb{R}}

\newcommand{\XX}{\mathfrak{X}}
\newcommand{\SSS}{\mathfrak{S}}
\newcommand{\TT}{\mathfrak{T}}
\newcommand{\UU}{\mathfrak{U}}
\newcommand{\VV}{\mathfrak{V}}
\newcommand{\YY}{\mathfrak{Y}}

\newcommand{\cc}{\mathbf{c}}
\newcommand{\CC}{\mathbf{C}}
\newcommand{\mm}{\mathbf{m}}

\newcommand{\MM}{\mathcal{M}}

\tikzset{dots/.append style={ultra thick, fill=none}}

\begin{document}
	
	\title[Maximal operators on Lorentz spaces in non-doubling setting]{Maximal operators on Lorentz spaces in non-doubling setting}
	
	\author{Dariusz Kosz}
	\address{ \newline Dariusz Kosz
		\newline Faculty of Pure and Applied Mathematics
		\newline Wroc\l aw University of Science and Technology 
		\newline Wybrze\.ze Wyspia\'nskiego 27, 50-370 Wroc\l aw, Poland
		\newline \textit{E-mail address:} \textnormal{dariusz.kosz@pwr.edu.pl}	
	}
	
	\thanks{This paper constitutes a part of PhD Thesis written in the Faculty of Pure and Applied Mathematics, Wroc\l aw University of Science and Technology, under the supervision of Professor Krzysztof Stempak.	
	}

	\begin{abstract} We study mapping properties of the centered Hardy--Littlewood maximal operator $\MM$ acting on Lorentz spaces $L^{p,q}(\XX)$ in the context of certain non-doubling metric measure spaces $\XX$. The special class of spaces for which these properties are very peculiar is introduced and many examples are given. In particular, for each $p_0, q_0, r_0 \in (1, \infty)$ with $r_0 \geq q_0$ we construct a space $\XX$ for which the associated operator $\MM$ is bounded from $L^{p_0,q_0}(\XX)$ to $L^{p_0,r}(\XX)$ if and only if $r \geq r_0$. 
	
	
	\end{abstract}
	
	\maketitle
	
	\section{Introduction} \label{S1}
	
	Maximal operators are objects of fundamental importance in mathematics, especially in harmonic analysis. In short, their main role is to estimate from above the values of many other operators under consideration. This in turn means that the standard way of using them should be somehow related to the property that they are bounded from one function space to another. In fact, there are hundreds of works that use various types of boundedness of maximal operators.
	
	Among the whole family of the aforementioned objects, particular attention is focused on the classical Hardy--Littlewood maximal operators which are introduced in the context of an arbitrary metric measure space $\XX$ and usually appear in the literature in two versions, centered $\MM$ (see the definition in Section~\ref{S2}) and non-centered $\widetilde{\MM}$. The first remark about these operators is that they are bounded on $L^\infty(\XX)$ with constant $1$. However, to indicate any other properties, one should rather know more about the structure of $\XX$.
	
	At the beginning, let us say a few words about the classical situation in which $\XX$ is simply $\RR^n$, $n \in \NN$, equipped with Lebesgue measure and the Euclidean metric. One of the most important results obtained in this particular case is that both operators, $\MM$ and $\widetilde{\MM}$, are of weak type $(1,1)$ which means that they are bounded from $L^1(\XX)$ to $L^{1, \infty}(\XX)$. This fact has several significant consequences including the Lebesgue differentiation theorem, a famous result in real analysis. Besides, keeping in mind that the operators are sublinear one can use the Marcinkiewicz interpolation theorem to prove their strong type $(p,p)$ estimate (that is, the boundedness on $L^p(\XX)$) for each $p \in (1, \infty)$. 
	
	Further studies in this field are focused, among other things, on determining the best constants in certain inequalities with the maximal function, including the weak type $(1,1)$ inequality in the first place (see \cite{Al1, Me1, Me2}). Also some articles have been devoted to the boundedness properties of $\MM$ and $\widetilde{\MM}$ in the context of various function spaces in which the regularity of functions is measured (see \cite{BCHP, HO, Ki, Ku, Ta}). Finally, an important direction of research is to analyze what happens with the particular properties of maximal operators if the underlying metric measure space changes.
	
	The standard tools used to show the weak type $(1,1)$ estimate for maximal operators are covering lemmas. At first glance, the possibility of using them depends mainly on the metric properties of a given space. To illustrate this let us mention that in the case of $\RR^n$ with the Euclidean metric a suitable covering argument provides that $\MM$ is of weak type $(1,1)$ in the case of any ``sensible" measure (one can choose here an arbitrary Radon measure, for example). However, the situation changes significantly if only $\MM$ is replaced by $\widetilde{\MM}$. Namely, it is possible to find a measure on $\RR^n$, $n \geq 2$, for which the associated non-centered operator $\widetilde{\MM}$ is not of weak type $(1,1)$. In fact, Sj\"ogren \cite{Sj} showed that this is the case for the standard two-dimensional Gaussian measure ${\rm d}\mu(x,y) = e^{-(x^2+y^2)/2} {\rm d}x {\rm d}y$ (see also an example given by Aldaz \cite{Al2}). 
	
	The last fact suggests that the condition on $\XX$ ensuring that most of the classical theory works should rather take into account both the associated metric and measure. Indeed, in the context of arbitrary metric measure spaces, the so-called doubling condition has been extensively used. Roughly speaking, it says that the measure of a given ball $B$ is comparable to the measure of $2B$, the ball concentric with $B$ and of radius two times that of $B$. In addition to many other results, it turned out that for any doubling space $\XX$ the associated operators $\MM$ and $\widetilde{\MM}$ both satisfy the weak type $(1,1)$ estimate. There were also a few concepts regarding the possibility of replacing the doubling condition with some weaker conditions (see \cite{H}, for example) or even eliminating it at all.
	
	Nazarov, Treil, and Volberg made a great contribution to developing harmonic analysis on arbitrary metric measure spaces. Their famous work \cite{NTV} contains valuable observations on how to deal with various important problems in this field without having the doubling condition in hand. It is particularly interesting for us that the modified centered maximal operator $\MM_3$ has been introduced there (the modification is that the measure of the ball $3B$ instead of $B$ occurs in the denominator in the definition of $\MM$). The key point was that in applications $\MM_3$ can often be successfully used in place of $\MM$, while it has much better general mapping properties.  
	
	In the following years, several articles treating the weak type $(1,1)$ inequality appeared in the context of the families of modified maximal operators, $\{\MM_\kappa \colon \kappa \geq 1\}$ and $\{\widetilde{\MM}_\kappa \colon \kappa \geq 1\}$ (see \cite{Sa, St1, Te}). As a result, it turned out that for any $\XX$ such that the measure of each ball is finite the associated operators $\MM_\kappa$ and $\widetilde{\MM}_\kappa$ are of weak type $(1,1)$ for $\kappa \in [2,\infty)$ and $\kappa \in [3, \infty)$, respectively. Moreover, these ranges are sharp since it has also been shown that there exist metric measure spaces such that $\MM_\kappa$ (respectively, $\widetilde{\MM}_\kappa$) is not of weak type $(1,1)$ for each $\kappa \in [1,2)$ (respectively, for each $\kappa \in [1,3)$). The suitable examples are given in \cite{Sa, St2} (see also \cite{SS}, where certain details justifying the correctness of the construction described in \cite{Sa} are given). Some further results regarding modified maximal operators were also obtained by the author \cite{Ko3}.
	
	A slightly different branch in the study of maximal operators was indicated by the previously mentioned work \cite{Al2}. Namely, this article initiated the program of searching spaces for which the mapping properties of the associated maximal operators are very specific. H.-Q. Li wrote a series of papers (see \cite{Li1, Li2, Li3}) in which the so-called cusp spaces have been introduced for this purpose. For example, in \cite{Li2} it is shown that for each fixed $p_0 \in (1, \infty)$ there exists a space $\XX$ for which the associated operator $\MM$ is of strong type $(p, p)$ if and only if $p > p_0$. 
	
	Then, the full characterization of possibilities concerning the question ``for what ranges of the parameter $p$ the operators $\MM$ and $\widetilde{\MM}$ satisfy the weak type and strong type $(p, p)$ inequalities?" was given by the author \cite{Ko}. It is worth noting here that the examples given in \cite{Ko} were created as a result of the development of the construction introduced earlier in \cite{St2}. Moreover, the main result of \cite{Ko} was strengthened in \cite{Ko2} were also the restricted weak type $(p,p)$ inequalities (or, in other words, the boundedness from $L^{p,1}(\XX)$ to $L^{p,\infty}(\XX)$) for maximal operators were taken under consideration.  
	
	Recall that the aforementioned spaces $L^p(\XX)$, $L^{p,\infty}(\XX)$ and $L^{p,1}(\XX)$ are located on the scale of Lorentz spaces $L^{p,q}(\XX)$. Thus, the natural way to extend the area of research described in the last two paragraphs is to study the boundedness of maximal operators acting on Lorentz spaces. In the case of $\RR^n$ and the classical Lorentz spaces some results allowing to describe the mapping properties of maximal operators in a more quantitative way has already been given (see \cite{AM, Sr}). However, to the author's best knowledge, there are no specific examples in the literature showing explicitly various peculiar behaviors of these operators in this context. Therefore, the aim of the present article is to introduce an appropriate class of metric measure spaces which provides the opportunity to generate a lot of such examples. For clarity, throughout the rest of this work we deal only with the centered Hardy--Littlewood maximal operator $\MM$, but we emphasize here that very similar analysis can also be done for $\widetilde{\MM}$ instead.
	
	The structure of the paper is as follows. In Section~\ref{S2} some basic definitions and notational conventions are given. In Section~\ref{S3} we state Theorem~\ref{T1} and Theorem~\ref{T2} which are the main results of this work. Section~\ref{S4} is entirely devoted to the formulation of the $\textit{space}$ $\textit{combining}$ $\textit{technique}$ which is an improved version of the method used in \cite{Ko3} in the context of modified operators. In Section~\ref{S5} and Section~\ref{S6} we present certain construction processes and the $\textit{test}$ $\textit{spaces}$ are described. Finally, in Section~\ref{S7} the proofs of Theorem~\ref{T1} and Theorem~\ref{T2} are given. 
	
	\section{Preliminaries} \label{S2}
	
	Let $\XX = (X, \rho, \mu)$ be a metric measure space with a metric $\rho$ and a Borel measure $\mu$. Throughout this article, unless otherwise stated, we assume that $(X, \rho)$ is bounded (that is, $\rm{diam}$$(X) := \sup\{\rho(x,y) : x, y \in X\}$ is finite) and $\mu(X) < \infty$. By $B(x,s) = B_\rho(x,s)$ we denote the open ball centered at $x \in X$ with radius $s>0$. If we do not specify the center point and the radius we write simply $B$. To avoid certain notational complications we also assume that the measure of each ball is strictly positive. According to this we define the $\textit{centered}$ $\textit{Hardy--Littlewood}$ $\textit{maximal}$ $\textit{operator}$, $\MM_\XX$, by
	\begin{displaymath}
	\MM_\XX f(x) := \sup_{s > 0} \frac{1}{\mu(B(x,s))} \int_{B(x,s)} |f| \, {\rm d}\mu, \qquad x \in X,
	\end{displaymath}
	where $f \colon X \rightarrow \mathbb{C}$ is any Borel function. Let us emphasize here that, in view of the equality $\MM_\XX f  = \MM_\XX |f|$, each time we study the behavior of $\MM_\XX$ later on in this paper we restrict our attention to functions $f \geq 0$.   
	
	Now we introduce the Lorentz spaces $L^{p,q}(\XX)$. For any Borel function $f \colon X \rightarrow \mathbb{C}$ we define the distribution function $d_f \colon [0, \infty) \rightarrow [0, \infty)$ by
	\begin{displaymath}
		d_f(t) := \mu( \{ x \in X : |f(x)| > t \}  ), 
	\end{displaymath}
	and the decreasing rearrangement $f^* \colon [0, \infty) \rightarrow [0, \infty)$ by
	\begin{displaymath}
	f^*(t) := \inf \{u > 0 : d_f(u) \leq t \}.
	\end{displaymath} 
	Then for any $p \in [1, \infty)$ and $q \in [1, \infty]$ the space $L^{p,q}(\XX)$ consists of those $f$ for which the quasi-norm $\|f\|_{p,q}$ is finite, where
	\begin{displaymath}
	\|f\|_{p,q} := \left\{ \begin{array}{rl}
	p^{1/q} \Big( \int_0^\infty \big( t \, d_f(t)^{1/p} \big)^q \frac{{\rm d}t}{t}   \Big)^{1/q}& \textrm{if }  q \in [1, \infty),   \\
	\sup_{t > 0} t \, d_f(t)^{1/p} & \textrm{if }  q = \infty, \end{array} \right. 
	\end{displaymath}
	or, equivalently,
	\begin{displaymath}
	\|f\|_{p,q} := \left\{ \begin{array}{rl}
	\Big( \int_0^\infty \big( t^{1/p} \, f^*(t) \big)^q \frac{{\rm d}t}{t}   \Big)^{1/q}& \textrm{if }  q \in [1, \infty),   \\
	\sup_{t > 0} t^{1/p} \, f^*(t) & \textrm{if }  q = \infty. \end{array} \right. 
	\end{displaymath}
	The second formula is valid also for $p = \infty$ (here we use the convention $t^{1/\infty} = 1$ for $t>0$). However, it turns out that $L^{\infty,q}$ is non-trivial only if $q=\infty$ since in each of the remaining cases it contains only the zero-function. Let us also note that one could consider $L^{p,q}(\XX)$ even for the wider range $p,q \in (0, \infty]$, but this is not the case of our study. 
	
	Many observations and details concerning Lorentz spaces are included in \cite{BS}, for example. For our purposes, it is instructive that one can estimate $\|f\|_{p,q}$ very precisely by calculating the values $d_f(2^k)$, $k \in \ZZ$. Furthermore, recall that for each $p \in [1, \infty]$ the space $L^{p,p}(\XX)$ coincides with the usual Lebesgue space $L^p(\XX)$ and hence we write shortly $\|f\|_p$ instead of $\|f\|_{p,p}$. 
	
	Throughout the article $C>0$ stands for a large constant whose value may vary from occurrence to occurrence. While studying the behavior of $\MM_\XX$ acting from $L^{p,q}(\XX)$ to $L^{p,r}(\XX)$, we allow $C$ to depend on the parameters $p$, $q$, and $r$, but not on any other factors, including the underlying metric measure space. For any Borel set $A \subset X$ its indicator function is denoted by $\chi_A$. We also use the convection that $[v, v) = (v, v] = \emptyset$ and $[v,v] = \{v\}$ for any $v \in \RR \cup \{\infty\}$.
	
	\section{Main Results} \label{S3}
	The aim of this paper, as it was announced at the end of Section~\ref{S1}, is to examine mapping properties of the centered Hardy--Littlewood maximal operator acting on Lorentz spaces. More precisely, 
	we will be interested in studying inequalities of the form
	\begin{equation}\label{1}
	\|\MM_\XX f \|_{p,r} \leq c(p,q,r,\XX) \|f\|_{p,q}, \qquad f \in L^{p,q}(\XX),
	\end{equation}
	which, for various parameters $p$, $q$, and $r$, may or may not hold, depending on the structure of $\XX$. Our goal is to construct plenty of metric measure spaces in order to observe that the sets of parameters for which \eqref{1} occurs can vary in many different ways. Before doing so, we indicate the exact range of parameters that will be taken into account later on. 
	
	We say that a triple $(p,q,r)$ is $\it admissible$ if one of the conditions below is satisfied 
	\begin{itemize}
		\item $p=q=1$ and $r \in [1, \infty]$,
		\item $p \in (1, \infty)$ and $1 \leq q \leq r \leq \infty$. 
	\end{itemize}
	Then, for a fixed admissible triple $(p,q,r)$ and a metric measure space $\XX$ we denote by $\cc(p,q,r,\XX)$ the smallest constant $c(p,q,r,\XX)$ such that \eqref{1} holds 
	(if there is no such constant, then we write $\cc(p,q,r,\XX) = \infty$). It is instructive to mention here that, for fixed $p \in [1, \infty)$, the case $\cc(p,q,r,\XX) < \infty$ is easier to meet for smaller values of $q$ and bigger values of $r$.
	
	The proposed range of parameters seems to be suitable for the following reasons. Firstly, the problem is trivial if $p = \infty$. Secondly, since there are some natural (usually proper) inclusions between Lorentz spaces and the function $\MM_\XX f$ is not smaller than $f$ in most ``sensible" settings, we omit the case $r < q$. Finally, as we will see in Remark~\ref{R1} below, the case $p = 1$ and $q \in (1, \infty]$ also turns out to be outside our area of interest. We remove here the restriction that the diameter of a given space is finite.    
	
	\begin{remark} \label{R1}
		Let $\XX = (X, \mu, \rho)$ be a metric measure space such that $\mu(X) < \infty$ and for any $\epsilon>0$ there exists a Borel set $A$ with $0 < \mu(A) < \epsilon$. Then for any $q \in (1, \infty]$ and $r \in [1, \infty]$ the associated maximal operator $\MM_\XX$ does not map $L^{1,q}(\XX)$ into $L^{1, r}(\XX)$.
	\end{remark}
	
	Indeed, fix $q \in (1, \infty)$ and $r \in [1, \infty]$ (we omit the case $q = \infty$ since the thesis is the stronger the smaller value of $q$ is). Let $(A_k)_{k \in \NN}$ be a sequence of pairwise disjoint Borel subsets of $X$ such that 
\begin{displaymath}
2^{-l_k - 1} < \mu(A_k) \leq 2^{-l_k},
\end{displaymath}
where $(l_k)_{k \in \NN}$ is a sequence of positive integers satisfying $l_{k+1} \geq l_k + 2$. Define
\begin{displaymath}
f(x) = \sum_{k=1}^\infty \frac{2^{l_k}}{k} \chi_{A_k}(x)  
\end{displaymath}
and observe that
\begin{displaymath}
\| f \|_{1, q} \leq C \, \Big( \sum_{k=1}^{\infty} \Big( \frac{2^{l_k}}{k} \, \mu(A_k) \Big)^q  \Big)^{1/q} \leq C \, \Big( \sum_{k=1}^{\infty} k^{-q} \Big)^{1/q} \leq C.
\end{displaymath}
On the other hand, $\mu(X) < \infty$ implies that for any $x \in X$ we have
\begin{displaymath}
\MM_\XX f(x) \geq \frac{\|f\|_1}{\mu(X)} \geq \frac{1}{\mu(X)} \, \sum_{k=1}^{\infty} \frac{1}{2k} = \infty
\end{displaymath}
and hence $\MM_\XX f$ is not an element of $L^{1,r}(\XX)$. \smallskip

Having described the range of parameters, we can formulate the main results of this paper. Namely, we will prove two theorems stated below. 

\begin{theorem} \label{T1}
Fix an admissible triple $(p_0, q_0, r_0)$. Then
\begin{itemize}
	\item there exists a (non-doubling) metric measure space $\UU$ such that $\cc(p_0,q_0,r, \UU ) = \infty$ for $r \in [q_0, r_0]$, while $\cc(p_0,q_0,r,\UU) < \infty$ for $r \in (r_0, \infty]$,
	\item there exists a (non-doubling) metric measure space $\VV$ such that $\cc(p_0,q_0,r,\VV) = \infty$ for $r \in [q_0, r_0)$, while $\cc(p_0,q_0,r,\VV) < \infty$ for $r \in [r_0, \infty]$.	
\end{itemize}
\end{theorem}

\begin{theorem} \label{T2}
Fix an admissible triple $(p_0, q_0, r_0)$ with $q_0 \in (1, \infty]$. Then there \linebreak exists a (non-doubling) metric measure space $\YY$ such that $\cc(p_0,1,r_0, \YY ) < \infty$ and \linebreak $\cc(p_0,q_0,r_0, \YY ) = \infty$. 
\end{theorem}

Two more remarks are in order here. 

\begin{remark} \label{R2}
For a given space $\XX = (X, \rho, \mu)$ define $\XX' = (X, \rho', \mu')$ by letting $\rho' = C_1 \rho$ and $\mu' = C_2 \mu$ for some constants $C_1, C_2 > 0$. Then for each admissible triple $(p,q,r)$ we have $\cc(p,q,r,\XX) = \cc(p,q,r,\XX')$.
\end{remark}

Indeed, one can easily see that replacing $\rho$ with $\rho'$ does not change anything since for any $x \in X$ the families $\{ B_\rho(x,s) : s > 0 \}$ and $\{ B_{\rho'}(x,s) : s > 0 \}$ coincide. Moreover, replacing $\mu$ with $\mu'$ makes that both quasi-norms in \eqref{1} are multiplied by $C_2^{1/p}$. 

\begin{remark} \label{R3}
Fix a space $\XX$ and an admissible triple $(p,1,r)$ with $p \in (1, \infty)$. Suppose that the inequality in \eqref{1} holds with some constant $c(p,1,r, \XX)$ for all functions $f$ of the form $f = \chi_A$ where $A \subset X$ is Borel. Then there exists a numerical constant $C_3 = C_3(p,r)$ independent of the choice of $\XX$ such that the inequality in \eqref{1} holds for all $f \in L^{p,1}(\XX)$ with $C_3 \cdot c(p,1,r, \XX)$ instead of $c(p,1,r, \XX)$.  
\end{remark}

The result for $r = \infty$ is well known and can be found in the literature (see \cite[Theorem 5.3, p. 231]{BS}). Moreover, careful reading of the proof in \cite{BS} reveals that the claim follows also for $r \in [1, \infty)$.

\section{Space combining technique} \label{S4}
	
	It will be very convenient to begin our studies with the description of a specific strategy which will be often used later on. Suppose that we start with a given sequence of metric measure spaces $(\XX_n)_{n \in \NN}$ such that the behavior of the functions $(p,q,r) \mapsto \cc(p,q,r,\XX_n)$ is known. Our goal is to use the spaces $\XX_n$ to create a new space, say $\XX = (X, \rho, \mu)$, for which $\cc(p,q,r,\XX)$ is in some sense comparable to $\sup_{n \in \NN} \, \cc(p,q,r,\XX_n)$. It turns out that $\XX$ may be built in a very transparent way under the additional assumption that each of the spaces $\XX_n$ consists of finitely many elements. We present the construction of $\XX$ below. 
	
	For each $n \in \NN$ let $\XX_n = (X_n, \rho_n, \mu_n)$ be a metric measure space consisting of finitely many elements. We introduce $\rho_n'$ and $\mu_n'$ by rescaling (if necessary) $\rho_n$ and $\mu_n$, respectively, in such a way that the following conditions are satisfied:
	
		\begin{enumerate}[label=(\alph*)]
			\item \label{a} the diameter of $X_n$ with respect to $\rho_n'$ does not exceed $1$,
			\item \label{b} $2 \mu_{n+1}'(X_{n+1}) \leq \mu_n'(\{x\})$ for every $x \in X_n$ and $n \in \NN$,
			\item \label{c} $\sum_{n=1}^{\infty} \mu_{n}'(X_{n}) = 1.$
		\end{enumerate} 
	Note that the condition \ref{c} is only for purely aesthetic reasons. For each $n \in \NN$ we denote $\XX_n' = (X_n, \rho_n', \mu_n')$ and notice that, according to Remark~\ref{R2}, the functions $(p,q,r) \mapsto \cc(p,q,r,\XX_n')$ and $(p,q,r) \mapsto \cc(p,q,r,\XX_n)$ coincide. We let $X := \bigcup_{n \in \NN} X_n$, assuming that $X_{n_1} \cap X_{n_2} = \emptyset$ for any $n_1 \neq n_2$. Finally, we define the metric $\rho$ on $X$ by
	\begin{displaymath}
	\rho(x,y) := \left\{ \begin{array}{rl}
	\rho_n'(x,y) & \textrm{if }  \{x,y\} \subset X_n \textrm{ for some } n \in \NN,   \\
	2 & \textrm{otherwise,} \end{array} \right. 
	\end{displaymath} 
	and the measure $\mu$ on $X$ by
	\begin{displaymath}
	\mu(E) := \sum_{n \in \NN} \mu_n'(E \cap X_n), \qquad E \subset X.
	\end{displaymath}
	
	In the following proposition we describe the aforementioned relation between $\cc(p,q,r,\XX)$ and $\cc(p,q,r,\XX_n)$, $n \in \NN$.
	 
	\begin{proposition} \label{P}
		Let $(\XX_n)_{n \in \NN}$ be a given sequence of metric measure spaces and assume that each of them consists of finitely many elements. Define $\XX$ as above. Then for each admissible triple $(p,q,r)$ there exists a numerical constant $\CC = \CC(p,q,r)$ independent of the choice of the spaces $\XX_n$ such that
		\begin{displaymath}
		\frac{1}{\CC} \, \sup_{n \in \NN} \, \cc(p,q,r,\XX_n) \leq \cc(p,q,r,\XX) \leq \CC \,\sup_{n \in \NN} \, \cc(p,q,r,\XX_n).
		\end{displaymath}  
	\end{proposition}
	
	\begin{proof}
		Note that the process of rescaling metrics and measures, which was used in the construction of $\XX$, does not affect the studied mapping properties of the associated maximal operators. Thus, without any loss of generality, we simply assume that the family $\{ \XX_n : n \in \NN\}$ satisfies the conditions \ref{a}--\ref{c}. 
		
		Let us fix an admissible triple $(p, q, r)$. First we show that 
		\begin{displaymath}
		\sup_{n \in \NN} \, \cc(p,q,r,\XX_n) \leq \cc(p,q,r,\XX).
		\end{displaymath}
		Indeed, assume that $\cc(p,q,r,\XX) < \infty$ and notice the following observation. If we take $f \in L^{p,q}(\XX_n)$ for some $n \in \NN$ and next we extend $f$ to $F \in L^{p,q}(\XX)$ by setting $F(x)=0$ for $ x \in X \setminus X_n$, then $\|F\|_{p,q} = \|f\|_{p,q}$ (here the symbol $\| \cdot \|_{p,q}$ refers to function spaces over different measure spaces). Moreover, in view of $(\rm a)$, we have $\MM_{\XX} F(x) = \MM_{\XX_n} f(x)$ for any $x \in X_n$. Consequently, the inequality $\|\MM_{\XX_n} f\|_{p,r} \leq  \cc(p,q,r,\XX) \|f\|_{p,q}$ follows from $\|\MM_{\XX} F\|_{p,r} \leq  \cc(p,q,r,\XX) \|F\|_{p,q}$.
		
		Next we show the estimate
		\begin{displaymath}
		\cc(p,q,r,\XX) \leq \CC \, \sup_{n \in \NN} \, \cc(p,q,r,\XX_n).
		\end{displaymath}
		It is worth noting here that
		\begin{displaymath}
		\sup_{n \in \NN} \, \cc(p,q,r,\XX_n) \geq \frac{1}{C},
		\end{displaymath}
		which can easily be shown by taking $f = \chi_{X_1} \in L^{p,q}(\XX_1)$. Now we fix $F \in L^{p,q}(\XX)$ and define $f_n \in L^{p,q}(\XX_n)$, $n \in \NN$, by restricting $F$ to $X_n$. By \ref{a} and \ref{c} we have
		\begin{displaymath}
		\MM_{\XX} F = \max \{\MM_{\rm loc} F, \MM_{\rm glob} F \},
		\end{displaymath}
		where $\MM_{\rm loc} F(x) := \MM_{\XX_n} f_n(x)$ for each $x \in X_n$, $n \in \NN$, and $\MM_{\rm glob} F(x) := \|F\|_1$ is a constant function. Hence, we can write
		\begin{displaymath}
		\|\MM_{\XX} F\|_{p,r} \leq C \, \big( \| \MM_{\rm loc} F\|_{p,r} + \| \MM_{\rm glob} F\|_{p,r} \big).
		\end{displaymath}
		To estimate $\| \MM_{\rm glob} F\|_{p,r}$ we proceed in the following way. If $q = 1$, then
		\begin{displaymath}
		\| \MM_{\rm glob} F\|_{p,r} \leq C \, \|F\|_1 = C \, \int_0^1 F^*(t) \, {\rm d}t \leq C \, \int_0^1 t^{1/p} \, F^*(t) \, \frac{{\rm d}t}{t} = C \, \|F\|_{p,1}.
		\end{displaymath}
		Next, if $q \in (1, \infty)$, then $p \in (1, \infty)$ and applying H\"older's inequality we obtain
		\begin{align*}
		\| \MM_{\rm glob} F\|_{p,r} \leq C \, \|F\|_1 
		& \leq C \, \Big( \int_0^1 \big( t^{1/p} \, F^*(t) \big)^q \, \frac{{\rm d}t}{t} \Big)^{1/q} \Big( \int_0^1 t^{(1/q - 1/p) \, q'} \, {\rm d}t \Big)^{1/q'} \\
		& \leq C \, \|F\|_{p,q} \, \Big( \int_0^1 t^{(1/q - 1/p) \, q'} \, {\rm d}t \Big)^{1/q'},
		\end{align*}
		where $q'$ satisfies $1/q + 1/q' = 1$. Since $(1/q - 1/p) \, q' > -1$ for $p  \in  (1, \infty)$, we are done. Finally, if $q = \infty$, then $p  \in  (1, \infty)$ and $r = \infty$, and therefore
		\begin{align*}
		\| \MM_{\rm glob} F\|_{p,\infty} = \|F\|_1 
		 \leq \sum_{k=1}^{\infty} F^*(2^{-k}) \, 2^{-k} 
		  = \sum_{k=1}^{\infty} F^*(2^{-k}) \, 2^{-k(1/p + 1/p')}
		  \leq \|F\|_{p, \infty} \, \sum_{k=1}^{\infty}  2^{-k/p'},
		\end{align*}
		where $p'$ satisfies $1/p + 1/p' = 1$.
		
		Now it remains to estimate $\| \MM_{\rm loc} F\|_{p,r}$. First we consider the case $r \neq \infty$. Suppose that $d_{\MM_{\rm loc} F}(t) > 0$ for some $t > 0$. Then \ref{b} implies that
		\begin{displaymath}
		2 \max_{n \in \NN} d_{\MM_{\XX_n} f_n}(t) \geq d_{\MM_{\rm loc} F}(t),
		\end{displaymath}
		and hence
		\begin{displaymath}
		\|\MM_{\rm loc} F\|_{p,r} \leq C \, \Big( \int_0^\infty \big(  t \, d_{\MM_{\rm loc} F}(t)^{1/p}  \big)^r \frac{{\rm d}t}{t}  \Big)^{1/r} \leq C \, \Big( \sum_{n=1}^{\infty }\int_0^\infty \big( t \,  d_{\MM_{\XX_n}f_n}(t)^{1/p} \big)^r \frac{{\rm d}t}{t}  \Big)^{1/r}.
		\end{displaymath}
		From the definition of $\cc(p,q,r,\XX_n)$, $n \in \NN$, we have
		\begin{align*}
		& \Big( \sum_{n=1}^{\infty }\int_0^\infty \big( t \, d_{\MM_{\XX_n}f_n}(t)^{1/p}  \big)^r \frac{{\rm d}t}{t}  \Big)^{1/r} \\
		 & \qquad \qquad \leq C \, \sup_{n \in \NN} \, \cc(p,q,r,\XX_n) \, \Big( \sum_{n=1}^{\infty } \Big( \int_0^\infty \big(  t \, d_{f_n}(t)^{1/p} \big)^q \frac{{\rm d}t}{t} \Big)^{r/q} \Big)^{1/r} \\
		 & \qquad \qquad \leq C \, \sup_{n \in \NN} \, \cc(p,q,r,\XX_n) \, \Big( \sum_{n=1}^{\infty } \int_0^\infty \big( t \, d_{f_n}(t)^{1/p} \big)^q \frac{{\rm d}t}{t} \Big)^{1/q} \\
		 & \qquad \qquad =  C \, \sup_{n \in \NN} \, \cc(p,q,r,\XX_n)\, \Big( \int_0^\infty t^q \Big( \sum_{n=1}^{\infty}  d_{f_n}(t)^{q/p} \Big) \frac{{\rm d}t}{t} \Big)^{1/q},
		\end{align*}
		where in the second inequality we used the fact that $q \leq r$. Applying \ref{b} again, we deduce that if $\sum_{n=1}^{\infty}  d_{f_n}(t)^{q/p} > 0$, then
		\begin{displaymath}
		\sum_{n=1}^{\infty}  d_{f_n}(t)^{q/p} \leq C d_{f_{n(t)}}(t)^{q/p},
		\end{displaymath}
		where $n(t) := \min\{n \in \NN : d_{f_n}(t) > 0 \}$. Set $t_0 := 0$ and for $n \in \NN$ define
		\begin{displaymath}
		t_n := t_{n-1} \vee \max\{\MM_{\rm loc} f(x) : x \in X_n\}.
		\end{displaymath}
		Since for each $n \in \NN$ and $t > 0$ we have $d_{f_n}(t) \leq d_{F}(t)$, we conclude that
		\begin{displaymath}
		\Big( \int_0^\infty t^q \Big( \sum_{n=1}^{\infty}  d_{f_n}(t)^{q/p} \Big) \frac{{\rm d}t}{t} \Big)^{1/q} \leq C \, \Big( \sum_{n=1}^{\infty } \int_{t_{n-1}}^{t_n} \big( t \, d_{f_n}(t)^{1/p} \big)^q \frac{{\rm d}t}{t} \Big)^{1/q} \leq C \, \|F\|_{p,q}. 
		\end{displaymath}
		Finally, consider the case $r = \infty$. By using \ref{b} we have
		\begin{align*}
		\|\MM_{\rm loc} F\|_{p, \infty} = \sup_{t > 0} \, t \, d_{\MM_{\rm loc} F}(t)^{1/p} 
		& \leq C \, \sup_{t > 0} \, \sup_{n \in \NN} \, t \, d_{\MM_{\XX_n} f_n}(t)^{1/p} \\
		& = C \, \sup_{n \in \NN} \, \|\MM_{\XX_n} f_n\|_{p, \infty} \\
		& \leq C \, \sup_{n \in \NN} \, \Big( \cc(p,q,\infty,\XX_n) \|f_n\|_{p,q} \Big) \\
		& \leq C \, \sup_{n \in \NN} \, \cc(p,q,\infty,\XX_n) \, \|F\|_{p,q}, 
		\end{align*}
		which completes the proof.	
	\end{proof}
	
	Let us note here that whenever we want to apply Proposition~\ref{P} later on, we omit the details related to the proper indexing of the component spaces. The only important fact is that each time we use countably many spaces.
	
	At the end of this section we indicate that each space $\XX$ obtained by using Proposition~\ref{P} is non-doubling. Indeed, fix $\epsilon > 0$ and let $n_0 = n_0(\epsilon)$ be such that $\mu(X_{n_0}) < \epsilon$. Then for any $x \in X_{n_0}$ we have $B(x,3/2) = X_{n_0}$ which implies $\mu(B(x,3/2)) < \epsilon$, while $\mu(B(x,3)) = \mu(X)$.
	
	\section{Test spaces of first type} \label{S5}
	
	In Section~\ref{S5} and Section~\ref{S6} we consider auxiliary structures called test spaces. 
	Each test space is a system of finitely many points equipped with a metric measure structure. Hence, we can use it as a component space in Proposition~\ref{P}. The spaces constructed in Section~\ref{S5} are used in the proof of Theorem~\ref{T1}, while the ones described in Section~\ref{S6} appear in the proof of Theorem~\ref{T2}. From now on we write $|E|$ instead of $\mu(E)$ for any Borel set $E$. 
	
	\subsection{Test spaces of first type for $\bf p > 1$} \label{S5.1}
	Fix $l \in \NN$ and take a non-decreasing sequence of positive integers $\mm = \mm(l) = (m_1, \dots, m_l) \in \NN^{l}$. Denote $M_0 = 0$ and $M_j = \sum_{i=1}^j m_i$ for $j = 1, \dots, l$. We introduce a test space of first type $\SSS = \SSS_{\mm} = (S, \rho, \mu)$ as follows. Set $S := \{x_0, x_1, \dots, x_{M_l}\}$, where all points $x_i$ are pairwise different. Define $\rho$ by letting 
	\begin{displaymath}
	\rho(x,y) := \left\{ \begin{array}{rl}
	0 & \textrm{if } x = y, \\
	1 & \textrm{if } x \neq y \textrm{ and } x_0 \in \{x,y\}, \\
	2 & \textrm{otherwise.} \end{array} \right.
	\end{displaymath} 
	Finally, define $\mu = \mu_\mm$ by letting $|\{x_0\}| := 1$ and $|\{x_i\}| := 2^j$ for each $M_{j-1} < i \leq M_j$, $j=1, \dots, l$. 
	
	Figure~\ref{F1} shows a model of the space $\SSS$. The solid line between two points indicates that the distance between them equals $1$. Otherwise the distance equals $2$.  
	
	\begin{figure}[H]
		\begin{tikzpicture}
		[scale=.7,auto=left,every node/.style={circle,fill,inner sep=2pt}]
		
		\node[label={[yshift=-1cm]$x_0$}] (m0) at (8,1) {};
		\node[label=$x_{1}$] (m1) at (5,3)  {};
		\node[label=$x_{2}$] (m2) at (6.5,3)  {};
		\node[label={[yshift=-0.30cm]$x_{M_l-1}$}] (m3) at (9.5,3)  {};
		\node[label={[yshift=-0.14cm]$x_{M_l}$}] (m4) at (11,3)  {};
		\node[dots, scale=2] (m5) at (8,3)  {...};
		
		\foreach \from/\to in {m0/m1, m0/m2, m0/m3, m0/m4}
		\draw (\from) -- (\to);
		\end{tikzpicture}
		\caption{The model of the space $\SSS$.}
		\label{F1}
	\end{figure}
	Note that we can explicitly describe any ball:
	\begin{displaymath}
	B(x_0,s) = \left\{ \begin{array}{rl}
	\{x_0\} & \textrm{for } 0 < s \leq 1, \\
	S & \textrm{for } 1 < s, \end{array} \right.
	\end{displaymath} 
	\noindent and, for $i \in \{1, \dots, M_l\}$,
	\begin{displaymath}
	B(x_i,s) = \left\{ \begin{array}{rl}
	\{x_i\} & \textrm{for } 0 < s \leq 1, \\
	\{x_0, x_i\} & \textrm{for } 1 < s \leq 2,  \\
	S & \textrm{for } 2 < s. \end{array} \right.
	\end{displaymath}
	
	In the following lemma we express the behavior of $\cc(p,q,r, \SSS)$ in terms of $\mm$.  
	
	\begin{lemma} \label{L1}
		Let $\SSS = \SSS_\mm$ be the metric measure space defined as above. Then for each admissible triple $(p,q,r)$ there is a numerical constant $\CC_1 = \CC_1(p, q, r)$ independent of the choice of the sequence $\mm$ such that: if $r \in [1, \infty)$, then
		\begin{displaymath}
		\frac{1}{\CC_1} \, \Big( \sum_{j=1}^{l} 2^{jr(-1 + 1/p)} \, m_j^{r/p}    \Big)^{1/r} \leq \cc(p,q,r,\SSS) \leq \CC_1 \Big( \sum_{j=1}^{l} 2^{jr(-1 + 1/p)} \, m_j^{r/p}    \Big)^{1/r},
		\end{displaymath} 
		and, if $r = \infty$, then
		\begin{displaymath}
		\frac{1}{\CC_1} \, \sup_{j= \in \{1, \dots, l\}} 2^{j(-1 + 1/p)} \, m_j^{1/p}  \leq \cc(p,q,r,\SSS) \leq \CC_1 \sup_{j \in \{1, \dots, l\}} 2^{j(-1 + 1/p)} \, m_j^{1/p}.
		\end{displaymath} 
	\end{lemma}
	
	\begin{proof}
	Fix an admissible triple $(p,q,r)$. 
	First we estimate $\cc(p,q,r,\SSS)$ from above. It is worth noting here that if $r \in [1, \infty)$, then
	\begin{displaymath}
	\Big( \sum_{j=1}^{l} 2^{jr(-1 + 1/p)} \, m_j^{r/p} \Big)^{1/r} \geq \sup_{j \in \{1, \dots, l\}} 2^{j(-1 + 1/p)} \, m_j^{1/p} \geq \frac{1}{2}.
	\end{displaymath}
	Take $f \in L^{p,q}(\SSS)$ such that $\|f\|_{p,q} = 1$. One can easily check that
	\begin{displaymath}
	\MM_{\SSS} f \leq \max \{f, 2 \MM_0 f, \MM_{\rm glob} f\},
	\end{displaymath}
	where $\MM_0 f(x_0) := 0$ and $\MM_0 f(x_i) := f(x_0) / 2^j$ for $M_{j-1} < i \leq M_j$, $j=1, \dots, l$, while $\MM_{\rm glob} f(x) := \|f\|_1 / |S|$ is a constant function. Therefore, we can write
	\begin{displaymath}
	\|\MM_{\SSS} f\|_{p,r} \leq C \, \big( \|f\|_{p,r} +  \| \MM_0 f\|_{p,r} + \| \MM_{\rm glob} f\|_{p,r} \big).
	\end{displaymath}
	Since $q \leq r$, we have $\|f\|_{p,r} \leq C \|f\|_{p,q} = C$. The inequality $\| \MM_{\rm glob} f\|_{p,r} \leq C \|f\|_{p,q}$ can be obtained in the same way as in the proof of Proposition~\ref{P} (the fact that $|S| \neq 1$ is irrelevant here). Thus, it remains to estimate $\| \MM_0 f\|_{p,r}$. Note that $\|f\|_{p,q} = 1$ implies that $f(x_0) \leq (q/p)^{1/q}$ if $q \in [1, \infty)$ and $f(x_0) \leq 1$ if $q=\infty$. We consider only $q \in [1, \infty)$ and the remaining case $q=\infty$ can be treated very similarly. Since $\mm$ is non-decreasing, we have 
	\begin{displaymath}
	d_{\MM_0 f} \Big( \Big(\frac{q}{p} \Big)^{1/q} \, 2^{-j-1} \Big) \leq \left\{ \begin{array}{rl}
	0 & \textrm{for } j \leq 0, \\
	m_j 2^{j+1}  & \textrm{for } j = 1, \dots, l-1,  \\
	m_l 2^{l+1} & \textrm{for } j \geq l, \end{array} \right.
	\end{displaymath} 
	which implies that
	\begin{displaymath}
	\| \MM_0 f\|_{p,r} \leq C \, \Big(  \sum_{j \in \ZZ} \Big( d_{\MM_0 f} \Big( \Big(\frac{q}{p} \Big)^{1/q} \, 2^{-j-1} \Big) \Big)^{r/p} \, 2^{-jr} 
	\Big) ^{1/r} \leq C \,\Big( \sum_{j=1}^{l} 2^{jr(-1 + 1/p)} \, m_j^{r/p}    \Big)^{1/r}
	\end{displaymath}
	in the case $r < \infty$, and
		\begin{displaymath}
		\| \MM_0 f\|_{p,r} \leq C \, \sup_{j \in \ZZ} \Big( d_{\MM_0 f} \Big( \Big(\frac{q}{p} \Big)^{1/q} \, 2^{-j-1} \Big) \Big)^{1/p} \, 2^{-j} 
		\leq C \, \sup_{j \in \{1, \dots, l\} } 2^{j(-1 + 1/p)} \, m_j^{1/p}
		\end{displaymath}
	in the case $r = \infty$. 
	Finally, to obtain the reverse inequality from the thesis it suffices to take $f_0 := \chi_{\{x_0\}}$ and calculate $\| \MM_\SSS f_0 \|_{p,r}$. We omit the details here.
	\end{proof}
	
	Before we introduce further constructions of spaces let us look at the expression
	\begin{equation}\label{2}
	\Big( \sum_{j=1}^{l} 2^{jr(-1 + 1/p)} \, m_j^{r/p}    \Big)^{1/r}
	\end{equation}
	that appears in the thesis of Lemma~\ref{L1}. Observe that if $p \in (1, \infty)$, then the factor $ 2^{jr(-1 + 1/p)}$ tends rapidly to $0$ as $j$ tends to $\infty$. Thus, for example, given $r_0 \in [1, \infty)$ we can find a non-decreasing sequence of positive integers $(m_j)_{j \in \NN}$ such that the series in \eqref{2} is uniformly bounded in $l$ if and only if $r > r_0$. Unfortunately, this idea does not work for $p = 1$ and hence we consider this case separately.
	
	\subsection{Test spaces of first type for $\bf p=1$} \label{S5.2}
	Fix $l \in \NN$ and take a non-decreasing sequence of positive integers $\mm' = \mm'(l) = (m'_1, \dots, m'_l)$ with $m'_1 = 1$. Next, associate with $\mm'$ an increasing sequence of positive integers $(h_1, \dots, h_j)$ such that
	\begin{equation}\label{3}
	\lfloor 2^{h_{j+1}} / m_{j+1}' \rfloor > 2^{h_j}, 
	\end{equation}
	holds for each $j = 1, \dots, l-1$ (here and later on the symbol ${\lfloor} \cdot {\rfloor}$ refers to the floor function). We introduce a test space of first type $\SSS' = \SSS'_{\mm'} = (S, \rho, \mu)$ as follows. Set 
	\begin{displaymath}
	S := \{x_0\} \cup \{x_{j,k} : k = 1, \dots, 2^{h_j}, \ j=1, \dots, l\}, 
	\end{displaymath}
	where all elements are pairwise different. We use some auxiliary symbols for certain subsets of $S$. Namely, we set $\tilde{S}_0 := \emptyset$ and $S_{l+1} := \emptyset$, and denote 
	\begin{displaymath}
	S_j := \{x_{j,k} : k = 1, \dots, 2^{h_j}  \}, \quad \tilde{S}_j := \{x_{j,k} : k = \lfloor 2^{h_j} / m_{j}' \rfloor + 1, \dots, 2^{h_j}  \},
	\end{displaymath}
	for $j = 1, \dots, l$ (notice that if $m_j'= 1$ for some $j$, then $\tilde{S}_j = \emptyset$). Then we define the metric $\rho$ determining the distance between two different elements $x, y \in S$ by the formula
	\begin{displaymath}
	\rho(x,y) := \left\{ \begin{array}{rl}
	1 & \textrm{if } x_0 \in \{x,y\} \textrm{ or } \{x,y\} \in \tilde{S}_{j-1} \cup S_j \textrm{ for some } j \in \{1, \dots, l\},  \\
	2 & \textrm{otherwise.} \end{array} \right. 
	\end{displaymath}
	Finally, we let $\mu$ be counting measure.
	
	Again we can explicitly describe any ball:
	\begin{displaymath}
	B(x_0,s) = \left\{ \begin{array}{rl}
	\{x_0\} & \textrm{for } 0 < s \leq 1, \\
	S & \textrm{for } 1 < s, \end{array} \right.
	\end{displaymath} 
	\noindent for $k = 1, \dots, \lfloor 2^{h_j}/m_j' \rfloor$, $j \in \{1, \dots, l\}$,
	\begin{displaymath}
	B(x_{j,k},s) = \left\{ \begin{array}{rl}
	\{x_{j,k}\} & \textrm{for } 0 < s \leq 1, \\
	\{x_0\} \cup \tilde{S}_{j-1} \cup S_j & \textrm{for } 1 < s \leq 2,  \\
	S & \textrm{for } 2 < s, \end{array} \right.
	\end{displaymath}
	\noindent and, for $k = \lfloor 2^{h_j}/m_j' \rfloor + 1, \dots, 2^{h_j}$, $j \in \{1, \dots, l\}$,
	\begin{displaymath}
	B(x_{j,k},s) = \left\{ \begin{array}{rl}
	\{x_{j,k}\} & \textrm{for } 0 < s \leq 1, \\
	\{x_0\} \cup \tilde{S}_{j-1} \cup S_j \cup S_{j+1}& \textrm{for } 1 < s \leq 2,  \\
	S & \textrm{for } 2 < s. \end{array} \right.
	\end{displaymath}
	
	In the following lemma we express the behavior of $\cc(1,1,r, \SSS')$ in terms of $\mm'$.  
	
	\begin{lemma} \label{L2}
		Let $\SSS'_{\mm'}$ be the metric measure space defined as above. Then for each $r \in [1, \infty]$ there is a numerical constant $\CC_1' = \CC_1'(r)$ independent of the choice of the sequence $\mm'$ such that: if $r \in [1, \infty)$, then
		\begin{displaymath}
		\frac{1}{\CC_1'} \, \Big( \sum_{j=1}^{l-1} (m_j')^{-r} \Big)^{1/r} \leq \cc(1,1,r,\SSS') \leq \CC_1' \Big( \sum_{j=1}^{l-1} (m_j')^{-r} \Big)^{1/r},
		\end{displaymath} 
		and, if $r = \infty$, then
		\begin{displaymath}
		\frac{1}{\CC_1'} \leq \cc(1,1,r,\SSS') \leq \CC_1' .
		\end{displaymath} 		
	\end{lemma}
	
	\begin{proof}
		Fix $r \in [1, \infty]$. First we estimate $\cc(1,1,r,\SSS')$ from above. 
		It is worth mentioning that if $r \in [1, \infty)$, then $m_1' = 1$ implies $\big( \sum_{j=1}^{l-1} (m_j')^{-r} \big)^{1/r} \geq 1$. We take $f \in L^1(\SSS')$ with $\|f\|_1 = 1$. One can easily check that
		\begin{displaymath}
		\MM_{\SSS'} f \leq \max \{f, \MM_0 f, \MM_{\rm glob} f\},
		\end{displaymath}
		where $\MM_0 f(x_0) := 0$ and $\MM_0 f(x) := |B(x,3/2)|^{-1}$
		for $x \in S \setminus \{x_0\}$, while $\MM_{\rm glob} f(x) := |S|^{-1}$ is a constant function. Therefore, we can write
		\begin{displaymath}
		\|\MM_{\SSS'} f\|_{1,r} \leq C \, \big( \|f\|_{1,r} +  \| \MM_0 f\|_{1,r} + \| \MM_{\rm glob} f\|_{1,r} \big).
		\end{displaymath}
		Obviously, we have $\|f\|_{1,r} \leq C \|f\|_1 = C$. As before, we can also prove that $\| \MM_{\rm glob} f\|_{1,r} \leq C \|f\|_1$. Thus, it remains to estimate $\| \MM_0 f\|_{1,r}$. By using \eqref{3} and the fact that $(h_1, \dots, h_l)$ is increasing we obtain 
		\begin{displaymath}
		d_{\MM_0 f} \big( 2^{i} \big) \leq \left\{ \begin{array}{rl}
		0 & \textrm{for } i \geq -h_1, \\
		C 2^{h_j} (m_j')^{-1} & \textrm{for } -h_{j+1} \leq i < -h_{j}, \ j = 1, \dots, l-1,  \\
		C 2^{h_l} & \textrm{for } i < -h_l, \end{array} \right.
		\end{displaymath} 
		which implies that
		\begin{displaymath}
		\| \MM_0 f\|_{1,r} \leq C \, \Big(  \sum_{i \in \ZZ} \big( d_{\MM_0 f} ( 2^{i} ) \big)^{r} \, 2^{ir} 
		\Big) ^{1/r} \leq C \,\Big( \sum_{j=1}^{l-1} (m_j')^{-r} \Big)^{1/r}
		\end{displaymath}
		in the case $r \in [1, \infty)$, and
		\begin{displaymath}
		\| \MM_0 f\|_{1,r} \leq C \, \sup_{i \in \ZZ} d_{\MM_0 f} \big( 2^{i} \big) \, 2^{i} \leq C
		\end{displaymath}
		in the case $r = \infty$. Finally, to obtain the reverse inequality from the thesis	it suffices to take $f_0 := \chi_{\{x_0\}}$ and calculate $\| \MM_{\SSS'} f_0 \|_{1,r}$. Again we omit the details.
	\end{proof}
	
	\section{Test spaces of second type} \label{S6}
	
	The aim of this section is to construct spaces $\XX$ such that, given an admissible triple $(p_0,q_0,r_0)$ with $q_0 \in (1, \infty]$, there exists a significant difference between the
	behavior of $\MM_{\TT}$ considered as an operator acting from $L^{p_0,q_0}(\XX)$ to $L^{p_0,r_0}(\XX)$ and from $L^{p_0,1}(\XX)$ to $L^{p_0,r_0}(\XX)$, respectively. 
	
	Let us begin with the following observation. Each space introduced in Section~\ref{S5} had one central point, namely $x_0$, and the function $\chi_{\{x_0\}}$ played the main role in estimating the size of $\cc(p,q,r,\SSS)$ (or $\cc(p,q,r,\SSS')$, respectively). Since the values $\|\chi_{\{x_0\}} \|_{p_0,1}$ and $\|\chi_{\{x_0\}} \|_{p_0,q_0}$ are comparable, we are now forced to change the strategy and introduce test spaces of another type, say $\TT$, for which the size of $\cc(p_0,q_0,r_0,\TT)$ will be calculated by testing the action of the associated maximal operator on certain more complicated functions. This can be done if we ensure that our space will have more central points grouped into several different types. The detailed analysis will be made separately for the following two cases: $1 < q_0 \leq r_0 < \infty$ and $1 < q_0 < r_0 = \infty$. We omit the case $q_0 = r_0 = \infty$.
	
	\subsection{Test spaces of second type for $\bf r < \infty$} \label{S6.1}
	
	Let $(p, q, r)$ be a fixed admissible triple with $1 < q \leq r < \infty$ and take $l \in \NN$. Associate with the quadruple $(p, q, r, l)$ four sequences of positive integers, $(m_i)_{i=1}^l$, $(h_i)_{i=1}^l$, $(\alpha_i)_{i=1}^l$, and $(\beta_i)_{i=1}^l$, with the following properties:
	\begin{enumerate}[label=(\roman*)]
		\item \label{i} $h_{i+1} / h_i \in \NN$,
		\item \label{ii} $m_{i+1} \geq 2 m_i h_i$,
		\item \label{iii} $1 \leq m_{i}^{1-p}h_i <2$,
		\item \label{iv} $l^{\frac{p}{(p-1)r}} \alpha_1 \geq 2 m_l h_l$,
		\item \label{v} $\alpha_{i+1} \geq 2 \alpha_i \beta_i$,
		\item \label{vi} $1 \leq \alpha_{i}^{1-p} \beta_i h_i <2$.	
	\end{enumerate}
	
	The sequences introduced above will determine the structure of the test space constructed in this section. Let us emphasize that the properties \ref{i}--\ref{vi} can be met simultaneously. Indeed, assume $h_1 = m_1 = 1$ and choose $m_2 \geq 2 m_1 h_1$ such that the set $\{h \in \NN : 1 \leq m_{2}^{1-p}h < 2 \}$ contains at least $h_1$ elements. Thus, it is possible to take $h_2$ for which the conditions $h_2 / h_1 \in \NN$ and $1 \leq m_{2}^{1-p}h_2 <2$ are satisfied simultaneously. Continue this way until the whole sequences $(m_i)_{i=1}^l$ and $(h_i)_{i=1}^l$ are chosen. Next, assume that $\alpha_1$ satisfies $l^{\frac{p}{(p-1)r}} \alpha_1 \geq 2 m_l h_l$ and $\alpha_1^{1-p} h_1 <2$. Take $\beta_1$ such that $1 \leq \alpha_{1}^{1-p} \beta_1 h_1 <2$. Choose $\alpha_{2} \geq 2 \alpha_1 \beta_1$ such that $\alpha_2^{1-p} h_2 < 2$ and take $\beta_2$ satisfying $1 \leq \alpha_{2}^{1-p} \beta_2 h_2 <2$. Continue this way until the whole sequences $(\alpha_i)_{i=1}^l$ and $(\beta_i)_{i=1}^l$ are chosen.
	
	Now we formulate a few thoughts that one should keep in mind later on:
	\begin{itemize}
		\item the sequences $(m_i)_{i=1}^l$ and $(\alpha_i)_{i=1}^l$ are used to define the associated measure, while $(h_i)_{i=1}^l$ and $(\beta_i)_{i=1}^l$ help to describe the number of elements of a given type,
		\item the property \ref{i} allows the set of points of a given type to be divisible into the appropriate number of equinumerous subsets,
		\item the properties \ref{i} and \ref{v} say that the sequences $(m_i)_{i=1}^l$ and $(\alpha_i)_{i=1}^l$ grow very fast; this fact results in large differences between the masses of points of different types, which in turn simplifies many calculations regarding the distribution function,
		\item the properties \ref{iii} and \ref{vi} are of rather technical nature; they are responsible for the balance between the number of points of a given type and the mass of each one of them,
		\item the property \ref{iv} says that the values $\alpha_1, \dots, \alpha_l$ are relatively large compared with $m_1, \dots, m_l$ and $h_1, \dots, h_l$ which makes that the points from the upper level (see Figure~\ref{F2}) have much greater masses than the ones from the lower level,
		\item the property \ref{iv} is the only property involving the parameter $l$.
	\end{itemize}
	
	We are ready to construct a test space of second type $\TT = \TT_{p,q,r,l} = (T, \rho, \mu)$. Set
	\begin{displaymath}
	T := \{x_{i,j}, \, x^\circ_{i,k} : i=1, \dots, l, \, j=1, \dots , h_i, \, k=1, \dots, h_i \beta_i\},
	\end{displaymath}
	where all elements $x_{i,j}, \, x^\circ_{i,k}$ are pairwise different. We use some auxiliary symbols for certain subsets of $T$: 
	\begin{displaymath}
	T^\circ := \{x^\circ_{i,k} : i=1, \dots, l, \, k=1, \dots, h_i \beta_i\},
	\end{displaymath}	
	for $i = 1, \dots, l$,
		\begin{displaymath}
		T_{i} := \{x_{i,j} : j=1, \dots, h_i \}, \quad
		T^\circ_{i} := \{x^\circ_{i,k} : k=1, \dots, h_i \beta_i\},
		\end{displaymath}
	and, for $1 \leq i \leq i^* \leq l$, $j=1, \dots, h_i$,  
	\begin{displaymath}
	T^\circ_{i^*,i,j} := \Big\{x^\circ_{i^*,k} : k \in \Big(\frac{j-1}{h_i} h_{i^*} \beta_{i^*} , \frac{j}{h_i} h_{i^*} \beta_{i^*} \Big]\Big\}.
	\end{displaymath}
	Observe that the sets $T^\circ_{i^*,i,j}$, $j = 1, \dots ,h_i$, are pairwise disjoint, each of them contains exactly $h_{i^*} \beta_{i^*}  / h_i$ elements (here the property \ref{i} was used) and $\bigcup_{j=1}^{h_i} T^\circ_{i^*,i,j} = T^\circ_{i^*}$. 
	
	We introduce $\mu$ by letting $|\{ x_{i,j} \}| := m_i$ and $|\{ x^\circ_{i,k} \}| := l^{\frac{p}{(p-1)r}} \alpha_i$. Note that, in view of \ref{ii}, \ref{iv}, and \ref{v}, the measure $\mu$ satisfies the following inequalities: for each $x \in T^\circ$,
	\begin{displaymath}
	|\{x\}| > |T \setminus T^\circ|,
	\end{displaymath}
	and, for each $1 \leq i < i^* \leq l$, $x_1 \in T_{i^*}$, and $x_2 \in T^\circ_{i^*}$,  
	\begin{displaymath}
	|\{x_1\}| > |T_1 \cup \ldots \cup T_i|, \quad  |\{x_2\}| > |T^\circ_1 \cup \ldots \cup T^\circ_i|.
	\end{displaymath}
	
	Finally, we define the metric $\rho$ on $T$ determining the distance between two different elements $x, y \in T$ by the formula
	\begin{displaymath}
	\rho(x,y) := \left\{ \begin{array}{rl}
	1 & \textrm{if } \{x, y\} = \{x_{i,j},x^\circ_{i^*,k}\} \textrm{ and } x^\circ_{i^*,k} \in T^\circ_{i^*,i,j},  \\
	2 & \textrm{otherwise.} \end{array} \right. 
	\end{displaymath}
	It is worth noting here that for each $1 \leq i \leq i^* \leq l$ and $x \in T^\circ_{i^*}$ there is exactly one point $y \in T_i$ such that $\rho(x,y)=1$.
	 
	Figure~\ref{F2} shows a model of the space $(T, \rho)$ for $l = 2$, $h_1 = 1$, and $h_2 = 2$.
	
	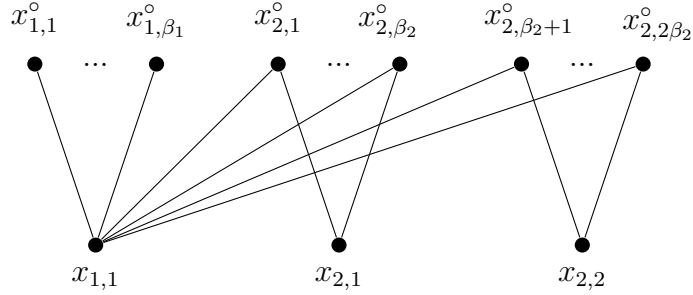
\begin{figure}[H] 
		\begin{tikzpicture}
		[scale=.8,auto=left,every node/.style={circle,fill,inner sep=2pt}]
		
		\node[label={[yshift=-1cm]$x_{1,1}$}] (c0) at (6,1) {};
		\node[label=$x^\circ_{1,1}$] (c1) at (5,4)  {};
		\node[label={[yshift=-0.1cm]$x^\circ_{1,\beta_1}$}] (c2) at (7,4)  {};
		\node[dots] (c3) at (6,4)  {...};
		
		\node[label={[yshift=-1cm]$x_{2,1}$}] (l0) at (10,1) {};
		\node[label=$x^\circ_{2,1}$] (l1) at (9,4)  {};
		\node[label={[xshift=-0.15cm, yshift=-0.1cm]$x^\circ_{2,\beta_2}$}] (l2) at (11,4)  {};
		\node[dots] (l3) at (10,4)  {...};
		
		\node[label={[yshift=-1cm]$x_{2,2}$}] (r0) at (14,1) {};
		\node[label={[xshift=0.1cm, yshift=-0.25cm]$x^\circ_{2,\beta_2+1}$}] (r1) at (13,4)  {};
		\node[label={[xshift=0.2cm, yshift=-0.2cm]$x^\circ_{2,2\beta_2}$}] (r2) at (15,4)  {};
		\node[dots] (r3) at (14,4)  {...};
		
		\foreach \from/\to in {l0/l1, l0/l2, c0/c1, c0/c2, r0/r1, r0/r2, l1/c0, l2/c0, r1/c0, r2/c0}
		\draw (\from) -- (\to);
		\end{tikzpicture}
		\caption{The model of the space $(T, \rho)$ for $l = 2$, $h_1 = 1$, and $h_2 = 2$.}
		\label{F2}
	\end{figure}
	As before, we will explicitly describe any ball: for $i=1, \dots , l, \, j=1, \dots , h_i$,
	\begin{displaymath}
	B(x_{i,j},s) = \left\{ \begin{array}{rl}
	\{x_{i,j}\} & \textrm{for } 0 < s \leq 1, \\
	\{x_{i,j}\} \cup \bigcup_{i^* \geq i} T_{i^*,i,j}& \textrm{for } 1 < s \leq 2,  \\
	T & \textrm{for } 2 < s, \end{array} \right.
	\end{displaymath} 
	\noindent and, for $i^*=1, \dots , l, \, k=1, \dots, h_i \beta_i$,
	\begin{displaymath}
	B(x^\circ_{i^*,k},s) = \left\{ \begin{array}{rl}
	\{x^\circ_{i^*,k}\} & \textrm{for } 0 < s \leq 1, \\
	\{x^\circ_{i^*,k}\} \cup \{x_{i, j} : x^\circ_{i^*,k} \in T_{i^*,i,j}\}& \textrm{for } 1 < s \leq 2,  \\
	T & \textrm{for } 2 < s. \end{array} \right.
	\end{displaymath}
	
	The following lemma describes the behavior of $\cc(p,1,r,\TT)$ and $\cc(p,q,r,\TT)$.
	
		\begin{lemma} \label{L3}
			Fix an admissible triple $(p,q,r)$ with $1 < q \leq r < \infty$ and take $l \in \NN$. Let $\TT = \TT_{p,q,r,l}$ be the metric measure space defined as above. Then there is a numerical constant $\CC_2 = \CC_2(p,q,r)$ independent of $l$ such that
			\begin{displaymath}
			 \cc(p,1,r,\TT) \leq \CC_2
			\end{displaymath} 
		and
			\begin{displaymath}
			\cc(p,q,r,\TT) \geq \frac{1}{\CC_2} l^{1-1/q}.
			\end{displaymath} 		
		\end{lemma}
		
		\begin{proof}
			First we estimate $\cc(p,1,r,\TT)$ from above. Let $f = \chi_E$ for some $\emptyset \neq E \subset T$. According to Remark~\ref{R3} it suffices to show that $\| \MM_\TT f \|_{p,r} \leq C \|f\|_{p,1}$ holds with some constant $C$ depending only on $p$ and $r$. Note that, by using sublinearity of $\MM_\TT$, we can assume that either $E \subset T^\circ$ or $E \subset T \setminus T^\circ$. 
			
			First we consider the case $E \subset T^\circ$. One can easily check that
			\begin{displaymath}
			\MM_{\TT} f \leq \max \{f, \chi_{T \setminus T^\circ}, |E| / |T|\}.
			\end{displaymath}
			Then the desired estimate follows easily from the fact that $|T \setminus T^\circ| < |E|$.
			
			Let us now consider the case $E \subset T \setminus T^\circ$. We have
			\begin{displaymath}
			\MM_{\TT} f \leq \max \{f, \MM_0 f, |E| / |T|\},
			\end{displaymath}
			where
			\begin{displaymath}
			\MM_0 f(x) := \chi_{T^\circ}(x) \, \frac{|E \cup B(x, 3/2)|}{|B(x, 3/2)|}.  
			\end{displaymath}
			It suffices to prove the estimate 
			\begin{equation}\label{4}
			\|\MM_0 f\|_{p,r}^p \leq C
			\|f\|_{p,1}^p = C |E|.
			\end{equation}
			Suppose for a moment that \eqref{4} holds for each $E$ satisfying $E \subset T_i$ for some $i \in \{1, \dots, l\}$. Then, for arbitrary $E \subset T \setminus T^\circ$, we define $E_i = E \cap T_i$ and $f_i = \chi_{E_i}$, $i = 1, \dots, l$, and observe that by \ref{ii}
			\begin{displaymath}
			\MM_0 f(x) \leq 2 \, \max_{i \in \{ 1, \dots, l\} } \MM_0 f_i(x)  
			\end{displaymath}
			for each $x \in T$. Therefore, for each $t>0$ we have $d_{\MM_0 f}(t) \leq \sum_{i=1}^l d_{\MM_0 f_i}(t/2)$ and hence
			\begin{displaymath}
			\|\MM_0 f \|_{p,r}^p \leq C \, \Big( \sum_{n \in \ZZ} 2^{nr} \big( 
			\sum_{i=1}^l d_{\MM_0 f_i} (2^n) \big)^{r/p}  
			\Big)^{p/r}.
			\end{displaymath}
			If $r \geq p$, then by Minkowski's inequality
			\begin{align*}
				\Big( \sum_{n \in \ZZ} 2^{nr} \big( 
				\sum_{i=1}^l d_{\MM_0 f_i} (2^n) \big)^{r/p}  
				\Big)^{p/r} \leq \sum_{i=1}^l  \Big( \sum_{n \in \ZZ} 2^{nr} \big( d_{\MM_0 f_i} (2^n) \big)^{r/p}  
				\Big)^{p/r} \leq
				C \sum_{i=1}^l |E_i| = C \|f\|_{p,1}^p.
			\end{align*}
			On the other hand, if $r < p$, then
			\begin{align*}
			\Big( \sum_{n \in \ZZ} 2^{nr} \big( 
			\sum_{i=1}^l d_{\MM_0 f_i} (2^n) \big)^{r/p}  
			\Big)^{p/r} \leq 
			\Big(  \sum_{i=1}^l \sum_{n \in \ZZ} 2^{nr}\big( d_{\MM_0 f_i} (2^n) \big)^{r/p} 
			\Big)^{p/r} 
			\leq C \Big( \sum_{i=1}^l |E_i|^{r/p} \Big)^{p/r} 
			\end{align*}
			and, in view of \ref{ii}, the last expression is estimated by $C |E|$ .
			
			Now we return to the proof of \eqref{4} for $E \subset T_i$. Suppose that $E$ consists of $\gamma$ elements for some $\gamma \in \{1, \dots, h_i\}$. Then for each $x \in T^\circ_{i^*}$, $i^* < i$, we have $\MM_0 f(x) = 0$. On the other hand, for each $i^* \geq i$ we have exactly $\gamma h_{i^*} \beta_{i^*} / h_i$ elements $x \in T^\circ_{i^*}$ for which
			\begin{displaymath}
			0 < \MM_0 f(x) \leq  l^{\frac{p}{(1-p)r}} m_i \alpha_{i^*}^{-1},
			\end{displaymath}
			while for the remaining points $x \in T^\circ_{i^*}$ we have $\MM_0 f(x) = 0$. Thus, we obtain
			\begin{align*}
			\| \MM_0 f \|_{p,r}^r & \leq C \, \sum_{i^* = i}^l \big( l^{\frac{p}{(1-p)r}} m_i \alpha_{i^*}^{-1}\big)^r \, \big( l^{\frac{p}{(p-1)r}} \alpha_{i^*} \gamma h_{i^*} \beta_{i^*} h_{i}^{-1} 
			\big)^{r/p} \\ 
			& \leq C \, \gamma^{r/p} m_i^r h_i^{-r/p} l^{-1} \sum_{i^* = i}^l \big( \alpha_{i^*}^{1-p} \beta_{i^*} h_{i^*} \big)^{r/p},   
			\end{align*}
			which is bounded by $C \gamma^{r/p} m_i^{r/p} = C |E|^{r/p}$ in view of \ref{iii} and \ref{vi}, and the fact that the sum in the expression above has at most $l$ elements. 
			
				In the next step we estimate $\cc(p, q, r, \TT)$ from below. Let $g$ be defined by
				\begin{displaymath}
				g := \sum_{i=1}^l \frac{1}{m_i} \, \chi_{T_i}.
				\end{displaymath}
				Then, by using \ref{ii}, we have
				\begin{align*}
				\|g\|_{p,q} \leq C \, \Big( \sum_{i=1}^l m_i^{-q} |T_i|^{q/p}  \Big)^{1/q} = C \, \Big( \sum_{i=1}^l (m_i^{1-p} h_i)^{q/p}  \Big)^{1/q}
				\end{align*}
				and \ref{iii} implies that $\|g\|_{p,q} \leq C \, l^{1/q}$.
				
				Now we focus on $\MM_\TT g$. For each $x \in T^\circ_i$, $i \in \{1, \dots, l\}$, we have 
				\begin{displaymath}
				\MM_\TT g(x) \geq \frac{1}{|B(x,3/2)|} \, \sum_{y \in B(x,3/2)} g(y) \, |\{y\}|.
				\end{displaymath}
				Note that \ref{iv} implies $|B(x,3/2)| \leq 2 l^{\frac{p}{(p-1)r}} \alpha_i$ and, as a result, we obtain
				\begin{displaymath}
				\MM_\TT g(x) \geq \frac{1}{C} \, l^{\frac{p}{(1-p)r}} \alpha_i^{-1} \, \sum_{i^*=1}^{i} \frac{1}{m_i} \, m_i = \frac{i}{C} \, l^{\frac{p}{(1-p)r}} \alpha_i^{-1}.
				\end{displaymath}
				Therefore, by using \ref{v}, we have
				\begin{displaymath}
				\|\MM_\TT g\|_{p,r}^r \geq \frac{1}{C} \, \sum_{i=1}^l \big(  i \, l^{\frac{p}{(1-p)r}} \alpha_i^{-1} 
				\big)^{r} |T^\circ_i|^{r/p} = \frac{1}{C} \, l^{-1} \, \sum_{i=1}^l i^r \, \big(\alpha_i^{1-p} \beta_i h_i\big)^{r/p} .
				\end{displaymath}
				In view of \ref{vi} each element of the series above is bigger than $i^r$ and hence
				\begin{displaymath}
				\|\MM_\TT g\|_{p,r} \geq \frac{1}{C} \, l^{-1/r} \, \big( \sum_{i=1}^l i^r \big)^{1/r} \geq \frac{1}{C} \, l^{-1/r + (r+1)/r} = \frac{l}{C}.
				\end{displaymath}
				Finally, the estimates for $\|g\|_{p,q}$ and $\|\MM_\TT g\|_{p,r}$ imply that $\cc(p,q,r,\TT) \geq l^{1-1/q} /C$.					
		\end{proof}
		
	\subsection{Test spaces of second type for $\bf r = \infty$} \label{S6.2}
	The argument described above needs only a few minor modifications to cover the second case under consideration. Let $(p,q,r)$ be a fixed admissible triple with $1 < q < r = \infty$ and take $l \in \NN$. Now we associate with $(p,q,r,l)$ a large constant $\alpha$ and three sequences of positive integers, $(m'_i)_{i=1}^l$, $(h'_i)_{i=1}^l$, and $(\beta'_i)_{i=1}^l$, with the following properties:
	\begin{enumerate}[label=(\subscript{\rm \roman*}{'})]
		\item \label{i'} $h'_{i+1} / h'_i \in \NN$,
		\item \label{ii'} $m'_{i+1} \geq 2 m'_i h'_i$,
		\item \label{iii'} $1 \leq (m'_i)^{1-p}h'_i < 2$,
		\item \label{iv'} $\alpha \geq 2 m'_l h'_l$,
		\item \label{v'} $j^{p-2} \leq \alpha^{1-p} \beta'_j h'_j \leq 2j^{p-2}$.
	\end{enumerate}
	Notice that the properties \ref{i'}--\ref{v'} can be met simultaneously. 
	
	We construct a space $\TT' = \TT'_{p,q,l} = (T, \rho, \mu)$ as follows. The set $T$ and the metric $\rho$ are defined as before with an aid of $(h'_i)_{i=1}^l$ and $(\beta'_i)_{i=1}^l$ instead of $(h_i)_{i=1}^l$ and $(\beta_i)_{i=1}^l$, respectively. Then we introduce $\mu$ by letting $|\{ x_{i,j} \}| := m'_i$ and $|\{ x_{i,k} \}| := i \alpha$.
	 
	 The following lemma describes the behavior of $\cc(p,1,\infty,\TT')$ and $\cc(p,q,\infty,\TT')$.
	 
	 \begin{lemma} \label{L4}
	 	Fix an admissible triple $(p,q,\infty)$ with $q \in (1, \infty)$ and take $l \in \NN$. Let $\TT'=\TT'_{p,q,l}$ be the metric measure space defined as above. Then there is a numerical constant $\CC'_2 = \CC'_2(p,q)$ independent of $l$ such that
	 	\begin{displaymath}
	 	\cc(p,1,\infty,\TT') \leq \CC'_2
	 	\end{displaymath} 
	 	and
	 	\begin{displaymath}
	 	\cc(p,q,\infty,\TT') \geq \frac{1}{\CC'_2} l^{1-1/q}.
	 	\end{displaymath} 		
	 \end{lemma}
	 
	 \begin{proof}
	 	We present only a sketch of the proof. First we want to estimate $\cc(p,1,\infty,\TT')$ from above. The main step here, as in the proof of Lemma~\ref{L3}, is to obtain
	 	\begin{displaymath}
	 	\| \MM_0  \chi_E \|_{p, \infty}^p \leq C \, \|\chi_E\|_{p,1}^p = C |E|,
	 	\end{displaymath}
	 	for $E \subset T_i$, $i \in \{1, \dots, l\}$, where
	 	\begin{displaymath}
	 	\MM_0 \chi_E(x) := \chi_{T^\circ}(x) \, \frac{|E \cup B(x, 3/2)|}{|B(x, 3/2)|}.  
	 	\end{displaymath}
	 	Suppose that $E$ consists of $\gamma$ elements for some $\gamma \in \{1, \dots, h'_i\}$. Then for each $x \in T^\circ_{i^*}$, $i^* < i$, we have $\MM_0 f(x) = 0$. On the other hand, for each $i^* \geq i$ we have exactly $\gamma h'_{i^*} \beta'_{i^*} / h'_i$ elements $x \in T^\circ_{i^*}$ for which
	 	\begin{displaymath}
	 	0 < \MM_0 f(x) \leq  m'_i \, (i^*)^{-1} \, \alpha^{-1},
	 	\end{displaymath}
	 	while for the remaining points $x \in T^\circ_{i^*}$ we have $\MM_0 f(x) = 0$. Thus, for each $i^* \geq i$,
	 	\begin{align*}
	 	\Big( \frac{m'_i}{i^* \alpha }\Big)^p \, \sum_{j=i}^{i^*} j \alpha \gamma \beta'_j h'_j (h'_1)^{-1} = \gamma (m'_i)^p (h'_i)^{-1} (i^*)^{-p}  \sum_{j=i}^{i^*} j \alpha^{1-p} \beta'_j h'_j, 
	 	\end{align*}
	 	which, in view of \ref{iii'} and \ref{v'}, is bounded by
	 	\begin{displaymath}
	 	C \gamma m'_i (i^*)^{-p}  \sum_{j=i}^{i^*} j^{p-1} \leq C \gamma m'_i = C |E|.
	 	\end{displaymath}
	 	
	 	Now we estimate $\cc(p, q, \infty, \TT')$ from below. Let $g$ be defined by
	 	\begin{displaymath}
	 			g := \sum_{i=1}^l \frac{1}{m'_i} \, \chi_{T_i}.
	 	\end{displaymath}
	 	Then, by using \ref{ii'} and \ref{iii'}, we have
	 	\begin{align*}
	 	\|g\|_{p,q} \leq C \, \Big( \sum_{i=1}^l (m'_i)^{-q} |T_i|^{q/p}  \Big)^{1/q} = C \, \Big( \sum_{i=1}^l \big( (m'_i)^{1-p} h'_i \big)^{q/p}  \Big)^{1/q} \leq C \, l^{1/q}.
	 	\end{align*}
	 	Observe that, by using \ref{iv'}, for each $x \in T^\circ_i$, $i \in \{1, \dots, l\}$, we have $\MM_{\TT'} g(x) \geq (2 \alpha)^{-1}$. Thus, in view of \ref{v'}, we obtain
	 	\begin{displaymath}
	 	\|\MM_{\TT'} g\|_{p,\infty}^p \geq \frac{1}{C} \, \alpha^{-p} \, |T^\circ| = \frac{1}{C} \, \sum_{i=1}^l \alpha^{1-p} j \beta'_j h'_j \geq \frac{1}{C} \, \sum_{i=1}^l j^{p-1} \geq \frac{l^p}{C}.
	 	\end{displaymath}
	 	Finally, the estimates for $\|g\|_{p,q}$ and $\|\MM_{\TT'} g\|_{p,\infty}$ imply that $\cc(p,q,\infty,\TT') \geq l^{1-1/q} /C$.    
	 \end{proof}
		
	\section{Proofs of Theorem~\ref{T1} and Theorem~\ref{T2}} \label{S7}
	
	Having introduced two types of test spaces, we are ready to prove Theorem~\ref{T1} and Theorem~\ref{T2}. We emphasize here that all spaces $\UU$, $\VV$ and $\YY$ that appear in the proofs are constructed by using Proposition~\ref{P} and hence they are non-doubling in view of the remark made at the end of Section~\ref{S4}.
	
	\begin{proof1*}
	Let $(p_0, q_0, r_0)$ be a fixed admissible triple. We consider four cases depending on the values of $p_0$ and $r_0$.
	
	\smallskip \noindent \bf Case~1: \rm $p_0 \in (1, \infty)$ and $r_0 \in [1, \infty)$. First we obtain a space $\UU$ such that $\cc(p_0,q_0,r, \UU ) = \infty$ for $r \in [q_0, r_0]$, while $\cc(p_0,q_0,r,\UU) < \infty$ for $r \in (r_0, \infty]$. Define $(a_i )_{i \in \NN}$ by the formula $a_i :=  2^{i(p_0-1)} i^{-p_0/r_0}$ and let $i_0 \in \NN$ be the first index such that $a_{i_0+1} \geq 1$ and $(a_i)_{ i = i_0+1}^\infty$ is non-decreasing. Thus, the sequence $(\bar{a}_i)_ {i \in \NN}$ defined by
	\begin{displaymath}
	\bar{a}_i := \left\{ \begin{array}{rl}
	1 & \textrm{for } 1 \leq i \leq i_0, \\
	\lceil a_i \rceil & \textrm{for } i > i_0, \end{array} \right.
	\end{displaymath}
	is also non-decreasing (here the symbol ${\lceil} \cdot {\rceil}$ refers to the ceiling function). Then, for any $n \in \NN$ let $\SSS_n$ be the test space $\SSS_\mm$ from Section~\ref{S5.1} with $l = n$ and $\mm = (\bar{a}_1, \dots, \bar{a}_n )$. We let $\UU$ be the space obtained by using Proposition~\ref{P} for the family $\{ \SSS_n : n \in \NN\}$. It is not hard to see that $\UU$ satisfies the desired properties. Indeed, fix $r \in [1, \infty)$. By using Lemma~\ref{L1} we have that $\cc(p_0,q_0,r,\UU)$ is comparable to 
	\begin{displaymath}
	\sup_{n \in \NN} \Big( \sum_{i=1}^{n} 2^{ir(-1 + 1/p_0)} \, \bar{a}_i^{r/p_0}    \Big)^{1/r},
	\end{displaymath}
	which, in turn, is comparable to
	\begin{equation}\label{5}
	\sup_{n \in \NN, \, n > i_0} \Big( \sum_{i=1}^{i_0} 2^{jr(-1 + 1/p_0)} +  \sum_{i=i_0+1}^{n}  i^{-r/r_0} \Big)^{1/r}.
	\end{equation}
	We can easily see that the second series in \eqref{5} tends to $\infty$ as $n \rightarrow \infty$ if and only if $r \leq r_0$. 
	
	Finally, a slight modification of the above argument allows us to get a space $\VV$ such that $\cc(p_0,q_0,r,\VV) = \infty$ for $r \in [q_0, r_0)$, while $\cc(p_0,q_0,r,\VV) < \infty$ for $r \in [r_0, \infty]$. Namely, instead of $(a_i)_{ i \in \NN }$ we will use a family of sequences $ \{ ( a_i^{(n)} )_{ i \in \NN} : n \in \NN \}$, where $a_i^{(n)} := a_i \, \log(n+3)^{-p_0/r_0}$. Then for each $n \in \NN$ we build $(\bar{a}_i^{(n)})_{ i \in \NN}$ in such a way as before with an aid of $(a_i^{(n)})_{ i \in \NN}$ and the critical index $i_0^{(n)}$. After all, we let $\SSS_n$ be the test space $\SSS_\mm$ with $l = n$ and $\mm = ( \bar{a}_1^{(n)}, \dots, \bar{a}_n^{(n)})$. It is clear that $\cc(p_0,q_0,r_0,\VV)$ is estimated from above by
	\begin{displaymath}
	C \, \sup_{n \in \NN} \Big( \sum_{i=1}^{\infty} 2^{jr_0(-1 + 1/p_0)} +  \sum_{i=1}^{n}  i^{-1} \, \log(n+3)^{-1} \Big)^{1/r_0} < \infty,
	\end{displaymath}
	and thus $\cc(p_0,q_0,r,\VV) < \infty$ for $r \in [r_0, \infty]$. Now we take $r \in [q_0, r_0)$. Since for every $n \in \NN$ the sequence $ ( a_i^{(n)})_{ i = i_0 +1}^\infty$ is non-decreasing, we have $ \bar{a}_i^{(n)} \geq a_i^{(n)}$ for every $i > i_0$ and $n \in \NN$. 
	Thus, we can estimate $\cc(p_0,q_0,r,\VV)$ from below by
	\begin{displaymath}
	\frac{1}{C} \, \sup_{n \in \NN \setminus \{1, \dots, i_0\}} \Big(\sum_{i=i_0+1}^{n}  i^{-r/r_0} \, \log(n+3)^{-r/r_0} \Big)^{1/r_0},
	\end{displaymath}
	which in fact is equal to $\infty$.
	
	\smallskip \noindent \bf Case~2: \rm $p_0 \in (1, \infty)$ and $r_0 = \infty$. Let $(b_i)_{ i \in \NN}$ be a non-decreasing sequence of positive integers such that $ \lim_{i \rightarrow \infty} 2^{i(-1+1/p_0)} b_i^{1/p_0} =  \infty$. Then for each $n \in \NN$ we let $\SSS_n$ be the test space $\SSS_\mm$ with $l = n$ and $\mm = (b_1, \dots, b_n )$. Finally, we let $\UU$ be the space obtained by using Proposition~\ref{P} for the family $\{ \SSS_n : n \in \NN\}$. It is routine to check that $\cc(p_0,q_0,r,\UU) = \infty$ holds for every $r \in [q_0, \infty]$.
	
	In order to obtain $\VV$ such that $\cc(p_0,q_0,r,\VV) = \infty$ for $r \in [q_0, \infty)$, while $\cc(p_0,q_0,\infty,\VV) < \infty$ we use some kind of diagonal argument. Namely, consider a sequence $(r^{(i)})_{i \in \NN}$ such that $r^{(i)} \in [1, \infty)$, $i \in \NN$, and $\lim_{i \rightarrow \infty} r^{(i)} = \infty$. Then for each $i \in \NN$ let $(\SSS_n^i)_{n \in \NN}$ be the sequence of test spaces $\SSS_\mm$ used in Case~1 to build $\UU$ for $r_0 = r^{(i)}$. Now we construct $\VV$ by using Proposition~\ref{P} for the whole family $\{\SSS_n^i : n, \, i \in \NN  \}$. For every $r < \infty$ there is $i_0 \in \NN$ such that $r^{(i_0)} > r$, which implies that $\cc(p_0,q_0,r,\VV)$ is estimated from below by
	\begin{displaymath}
	\sup_{n \in \NN} \cc(p_0,q_0,r,\SSS_n^{i_0})  = \infty.
	\end{displaymath}
	On the other hand, it is not hard to see that for each $n \in \NN$ and $i \in \NN$ we have $\cc(p_0,q_0,\infty,\SSS_n^{i}) \leq C$, which implies that $\cc(p_0,q_0,\infty,\VV) < \infty$.  
	
	\smallskip \noindent \bf Case~3: \rm $p_0 = 1$ and $r_0 \in [1, \infty)$. First we obtain a space $\UU$ such that $\cc(1,1,r, \UU ) = \infty$ for $r \in [1, r_0]$, while $\cc(1,1,r,\UU) < \infty$ for $r \in (r_0, \infty]$. Consider the non-decreasing sequence $(c_i)_{ i \in \NN}$ defined by the formula $c_i := \lfloor i^{1/r_0} \rfloor$. Then, for any $n \in \NN$ let $\SSS'_n$ be the test space $\SSS'_{\mm'}$ from Section~\ref{S5.2} with $l = n$ and $\mm' = (c_1, \dots, c_n)$. We let $\UU$ be the space obtained by using Proposition~\ref{P} for the family $\{ \SSS'_n : n \in \NN\}$. Again, it is not hard to see that $\UU$ satisfies the desired properties. Indeed, fix $r \in [1, \infty)$. By using Lemma~\ref{L2} we have that $\cc(1,1,r,\UU)$ is comparable to 
	$\sup_{n \in \NN} \big( \sum_{i=1}^{n} i^{-r/r_0} \big)^{1/r}$ which is equal to $\infty$ if and only if $r \leq r_0$. 
	
	Now we build $\VV$ such that $\cc(1,1,r,\VV) = \infty$ for $r \in [1, r_0)$, while $\cc(1,1,r,\VV) < \infty$  for $r \in (r_0, \infty]$. For each $n \in \NN$ let $ ( c_i^{(n)})_{ i \in \NN}$ be defined by $ c_i^{(n)} := \lfloor i^{1/r_0} \log(n+3)^{1/r_0} \rfloor$. We let $\SSS'_n$ be the test space $\SSS'_{\mm'}$ with $n = l$ and $\mm' = ( c_1^{(n)}, \dots, c_n^{(n)} )$. Then we construct $\VV$ by using Proposition~\ref{P} for the family $\{ \SSS_n' : n \in \NN\}$. By using Lemma~\ref{L2}, for each fixed $r \in [1, \infty)$ we obtain that $\cc(1,1,r,\VV)$ is comparable to 
	$\sup_{n \in \NN} \big( \sum_{i=1}^{n} i^{-r/r_0} \, \log(n+3)^{-r/r_0} \big)^{1/r}$ which is equal to $\infty$ if and only if $r < r_0$. 
	
	\smallskip \noindent \bf Case~4: \rm $p_0 = 1$ and $r_0 = \infty$. In order to obtain $\UU$ such that $\cc(1,1,r,\UU) = \infty$ for every $r \in [1, \infty]$ we proceed as in Case~2. Namely, we choose a non-decreasing sequence of positive integers $(d_i)_{ i \in \NN}$ such that $ \lim_{i \rightarrow \infty} d_i =  \infty$. Next, for any $n \in \NN$ we let $\SSS_n$ be the test space $\SSS_\mm$ with $l = n$ and $\mm = (d_1, \dots, d_n )$. Then we denote by $\UU$ the space obtained by using Proposition~\ref{P} for the family $\{ \SSS_n : n \in \NN \}$. By using Lemma~\ref{L1} we conclude that $\UU$ satisfies the desired properties.
	
	Finally, we can obtain $\VV$ such that $\cc(1,1,r,\VV) = \infty$ for $r \in [1, \infty)$, while $\cc(1,1,\infty,\VV) < \infty$ by using the family of spaces $\{ \SSS^{'i}_n : n \in \NN, \, i \in \NN \}$ introduced similarly as in Case~2, but this time with an aid of the properly chosen test spaces $\SSS'_{\mm'}$ used in Case~3. We skip the details here. $\raggedright \hfill \qed$ 	   
	\end{proof1*}
	
	\begin{proof2*}
Let $(p_0,q_0,r_0)$ be a fixed admissible triple with $q_0 \in (1, \infty)$ (we omit the case $q_0 = \infty$ since the thesis is the stronger the smaller value of $q_0$ is). Consider the case $r_0 \in [q_0, \infty)$. For each $n \in \NN$ let $\TT_n$ be the test space $\TT_{p,q,r,l}$ from Section~\ref{S6.1} with $(p,q,r) = (p_0,q_0,r_0)$ and $l = n$. Then we let $\YY$ be the space obtained by using Proposition~\ref{P} for the family $\{ \TT_n : n \in\NN \}$. By using Lemma~\ref{L3} we conclude that $\cc(p_0,1,r_0,\YY) < \infty$, while $\cc(p_0, q_0, r_0, \YY) = \infty$. On the other hand, if $r_0 = \infty$, then we let $\TT'_n$ be the test space $\TT'_{p,q,l}$ from Section~\ref{S6.2} with $(p,q) = (p_0,q_0)$ and $l = n$. Finally, we construct $\YY$ by using Proposition~\ref{P} for the family $\{ \TT'_n : n \in\NN \}$ and the thesis can be deduced from Lemma~\ref{L4}.
	$\raggedright \hfill \qed$
	\end{proof2*}
	
	We conclude our studies with a short comment. Theorem~\ref{T1} describes the situation in which the maximal operator acts on a single Lorentz space $L^{p_0,q_0}(\XX)$. The main concept behind Theorem~\ref{T2}, in turn, is to show some differences between the actions of $\MM_\XX$ on $L^{p_0,q_1}(\XX)$ and $L^{p_0,q_2}(\XX)$, respectively, for some $p_0 \in (1, \infty)$ and $1 \leq q_1 < q_2 \leq \infty$. However, in view of Remark~\ref{R3}, this task is much easier if $q_1 = 1$ and this is the only case in which this problem has been solved here. Thus, a natural question arises. Is it possible to construct $\XX$ for which there is a significant difference between the actions of $\MM_\XX$ on $L^{p_0,q_1}(\XX)$ and $L^{p_0,q_2}(\XX)$, respectively, if $p_0 \in (1, \infty)$ and $1 < q_1 < q_2 \leq \infty$? We announce that the answer is affirmative and, in fact, it is a simple consequence of the results stated in the author's forthcoming paper, where a slightly more general problem involving the actions of $\MM_\XX$ on $L^{p_0,q}(\XX)$ for fixed $p_0 \in (1, \infty)$ and varying $q \in [1, \infty]$ is considered.
	 
	\section*{Acknowledgements}
	I would like to express my deep gratitude to Professor Krzysztof Stempak
	for his suggestion to study the problem discussed in this article. I also thank him for insightful comments and continuous help during the preparation of the paper.
	
	Research was supported by the National Science Centre of Poland, project no. \linebreak
	2016/21/N/ST1/01496.

\end{document}